\documentclass[10pt,b5paper]{article}

\usepackage[T1]{fontenc}
\usepackage[utf8]{inputenc}
\usepackage[english]{babel}
\usepackage{amsmath,amssymb,amsthm}
\usepackage{lmodern}
\usepackage{geometry}
\usepackage{graphicx}
\usepackage{url}
\usepackage{comment}
\usepackage{titlesec}
\titleformat{\paragraph}[runin]{\normalfont\itshape}{}{0pt}{}[.]
\titlespacing*{\paragraph}{0pt}{0.7em}{0.5em}

\newtheorem{theorem}{Theorem}[section]
\newtheorem{lemma}[theorem]{Lemma}
\newtheorem{proposition}[theorem]{Proposition}
\newtheorem{corollary}[theorem]{Corollary}
\theoremstyle{definition}
\newtheorem{definition}[theorem]{Definition}
\theoremstyle{remark}
\newtheorem{remark}[theorem]{Remark}
\numberwithin{equation}{section}

\DeclareMathOperator{\E}{E}
\DeclareMathOperator{\Var}{Var}

\DeclareMathOperator{\Bin}{Bin}
\DeclareMathOperator{\Mult}{Mult}

\title{Discrete total variation under the multinomial, Part~1:\\
expectation under the uniform null}
\author{
N.~Bani\'c\\
Gideon Brothers, 10000 Zagreb, Croatia\\
\texttt{nbanic@gmail.com}
\and
N.~Elezovi\'c\\
Department of Applied Mathematics,\\
Faculty of Electrical Engineering and Computing,\\
University of Zagreb, 10000 Zagreb, Croatia\\
\texttt{neven.elezovic@fer.hr}
}
\date{}

\begin{document}
\maketitle

\begin{abstract}
We study the expectation of the discrete total variation statistic
$d=\sum_{i=1}^{n-1}|X_{i+1}-X_i|$ of counts $X_1,\dots,X_n$ over $n$ ordered
bins with sample size $N$, under the uniform multinomial null. In the notation
used throughout the series, a reference distribution $\mathbf r$ defines
$d_{\mathbf r}(x)=\|x-N\mathbf r\|_V$; here $\mathbf r=\mathbf u$ is uniform and
we write simply $d$. We obtain an
exact single-sum formula, a closed form in central trinomial coefficients valid
for every $n\ge2$, with one-term evaluations at $n=2,4$ and a complete D-finite
characterization for $n=3$. The per-edge
expectation $W_N=\E|X_2-X_1|$ satisfies an exact three-term recurrence, from
which we derive its complete $1/N$ expansion with the $n$-independent constant
$\E(d^\circ/\sqrt{Nn})\to 2/\sqrt\pi$ for the cyclic statistic
$d^\circ=d^\circ_{\mathbf u}$, independent of the number of bins; the
complementary sparse regime (fixed $N$, $n\to\infty$) is treated in a separate
paper. Conditioning on a two-bin total reduces each edge to a symmetric random
walk; as the subsequent parts of the series will show, this method carries the
analysis to non-uniform reference laws.
\end{abstract}

\medskip\noindent\textbf{Keywords:} discrete total variation; multinomial
distribution; goodness-of-fit; successive differences; central binomial
coefficient; central trinomial coefficient; asymptotic expansion.

\medskip\noindent\textbf{MSC 2020:} 60C05 (primary); 62E15; 62G10; 05A16; 41A60.

\section{Introduction}

	A great many practical questions reduce to a single one: \emph{is this histogram
	flat?} One observes counts $X_1,\dots,X_n$ across $n$ \emph{ordered} categories ---
	analog-to-digital converter codes, digit values, time bins --- and must decide whether
	they are consistent with the uniform law. The classical instrument is Pearson's
	chi-square test~\cite{pearson1900}: universal, well understood, and rightly the
	default. But it has one blind spot, invisible precisely because it is so fundamental:
	\emph{chi-square does not see order}. It treats the bins as an unordered set, so any
	permutation of the categories leaves its value unchanged. 
	This is harmless for many alternatives, but not for deviations that alternate from
	one bin to the next, $p_i=\tfrac1n+(-1)^i\varepsilon$ (a probability vector for even $n$; for odd $n$ one endpoint is adjusted). Such ``comb'' patterns occur,
	for example, in converter codes, digit preferences, and rounded or last-digit data.
	They are visible through neighboring differences, but not through an unordered
	sum of squared residuals alone.
	Chi-square is of course consistent against any fixed alternative, comb included;
	the issue is \emph{efficiency}. The comb's noncentrality is spread over $n-1$
	degrees of freedom, whereas a statistic that compares neighbors responds to it directly
	--- the difference that matters for local alternatives $\varepsilon\sim c/\sqrt N$
	and moderate $N$.

	The natural quantity that compares neighbors is the \emph{discrete total variation}
	(DTV),
	\[
		d=\sum_{i=1}^{n-1}|X_{i+1}-X_i|,
	\]
	the $\ell_1$ norm of the discrete derivative of the histogram. (Throughout, ``total variation'' refers to this $\ell^1$ norm of successive differences, not to the total-variation distance between measures.) As a functional,
	the general statistic used in the sequel is centered at a reference distribution
	$\mathbf r=(r_1,\ldots,r_n)$:
	\[
		d_{\mathbf r}(x)=\|x-N\mathbf r\|_V
		=\sum_{i=1}^{n-1}|(x_{i+1}-Nr_{i+1})-(x_i-Nr_i)|.
	\]
	The present paper treats the uniform reference
	$\mathbf u=(1/n,\ldots,1/n)$. In that case the centering cancels along every
	edge, so $d_{\mathbf u}(X)=\sum_{i=1}^{n-1}|X_{i+1}-X_i|$; throughout this
	paper we write simply $d$. Its cyclic analogue is denoted $d^\circ$
	(generally $d^\circ_{\mathbf r}$).
	Discrete total variation is standard in signal and image processing
	(total-variation denoising~\cite{rudin_osher_fatemi}); its \emph{statistical} use,
	as an outlier measure for histograms, was introduced in~\cite{banic_elezovic_tvor}
	and developed into the comb test in~\cite{banic_elezovic_ct}. It is
	order-sensitive, trivial to compute, and can achieve substantially higher power
	than chi-square against comb-like alternatives --- a claim documented empirically
	in~\cite{banic_elezovic_ct}; the present paper proves nothing about
	power and supplies the first moment of the null calibration.

	The statistic sits in a classical lineage. It is the $\ell^1$ analogue of the
	mean-square successive difference $\sum_i(X_{i+1}-X_i)^2$ of Abbe and
	von~Neumann~\cite{vonneumann1941,vonneumann_kent1941}, whose studentized form underlies the
	Durbin--Watson statistic~\cite{durbin_watson1950}; its $\ell^1$ analogue, the mean
	absolute successive difference, is a classical ordered-data statistic~\cite{kamat1953}.
	Among order-sensitive
	goodness-of-fit tools for ordered categories --- Neyman's smooth tests and their
	categorical developments~\cite{rayner_best}, the cumulative chi-square statistics
	of Hirotsu~\cite{hirotsu1986} and Nair~\cite{nair1987}, spacings
	tests~\cite{pyke1965} --- DTV is distinguished by what it does \emph{not} require.
	For the period-2 comb the two-sided contrast $|T|$, $T=\sum_i(-1)^iX_i$, is the
	locally optimal one-degree-of-freedom competitor and needs only the period; DTV
	needs neither the period nor the globality of the pattern --- it responds to any
	locally alternating or rough segment --- at the price of a composite, less
	concentrated alternative. Finally, DTV is the $\ell^1$ norm of the discrete
	\emph{derivative}, to be contrasted with the classical $\ell^1$ distance to
	uniformity $\sum_i|X_i-N/n|$, whose first-moment theory rests on the same
	de~Moivre-type identities that will drive our computations.

	A test statistic is only as useful as our knowledge of its null distribution, and here
	the price of order-sensitivity comes due. For the \emph{first} moment the difficulty
	is not the dependence of the counts: linearity of expectation and exchangeability
	reduce $\E(d)$ to a single two-bin marginal. The genuine obstacles at first-moment
	level are the non-smoothness of the absolute value and the asymptotics; the dependence
	through $\sum_iX_i=N$ and the overlap of neighboring summands $|X_{i+1}-X_i|$ begin
	to bite at the \emph{variance}, treated in Part~2, and the exact distribution is
	harder still. The exact null law of $d$ is already computable by dynamic
	programming~\cite{banic_elezovic_ct}, with a fitted gamma law calibrating beyond
	that; closed-form moments are what an analytic large-$N$ calibration needs.
	\emph{The present paper opens that program with the first moment under
	the uniform null}, the case where the theory is cleanest. The expectation has a closed form in
	central trinomial coefficients, valid for every $n\ge2$, and an $n$-independent
	limit $\E(d^\circ/\sqrt{Nn})\to2/\sqrt\pi$, together with its complete $1/N$ expansion --- governed by an
	exact three-term recurrence that delivers the coefficients in closed recursive form. Two facts
	drive every computation and are recalled in the Preliminaries: the edge variable
	$D_j=X_{j+1}-X_j$ depends only on a two-bin marginal, and that marginal splits into
	independent binomial steps; the same mixture-conditioning method also carries the
	analysis beyond the uniform case. We note at the outset that DTV is not a universal
	replacement for chi-square --- it gains over chi-square chiefly against period-2
	(alternating) departures, the gain fading as the period grows~\cite{banic_elezovic_ct},
	and is otherwise complementary to it.

	This paper is the first part of a larger study of exact and asymptotic moments of
	discrete total variation under multinomial sampling. Part~2 treats the variance
	under the uniform law; Parts~3 and~4 the expectation and variance under a general
	reference law $\mathbf r$, where, as Part~3 will show, mixture-conditioning still
	applies; Part~5 studies the cyclic
	statistic as its subject proper --- its variance, the exact linear--cyclic transfer
	identity, and the large-bin variance ridge --- whereas here $\E(d^\circ)$ follows from
	exchangeability at no extra cost. The distributional theory --- the asymptotic null
	laws of $d$ that complement the exact computation of~\cite{banic_elezovic_ct} ---
	is developed in a companion paper, for which the moments computed in this series
	are the inputs. Two regimes
	must be kept apart there: for \emph{fixed} $n$ as $N\to\infty$ the number of
	summands is fixed and $d/\sqrt N$ converges to a functional of finitely many
	dependent folded Gaussians, which is not normal in general; approximate normality
	of $d$ is a \emph{large-$n$} phenomenon, the folded summands being $1$-dependent
	in the Gaussian limit (exactly so after Poissonization), so an $m$-dependent
	central limit theorem~\cite{hoeffding_robbins1948} applies to their number $n$
	(cf.~Morris~\cite{morris1975} for the multinomial background). For clarity of
	scope: the present paper computes first moments under the uniform null only; it does
	not treat the variance (Part~2), non-uniform laws (Parts~3--4), the sparse regime of
	fixed $N$ and $n\to\infty$ (a separate paper), or the null distribution (the
	companion paper). Only the facts needed for the present expectation calculation are
	recalled below.

\section{Preliminaries: the multinomial and the edge variables}
	\label{sec:multinomial-basics}

	We collect the few facts about the multinomial law on which the expectation
	computations rest. Two carry the whole analysis: an edge $D_j=X_{j+1}-X_j$
	depends only on a two-bin marginal, and that marginal factors into independent
	binomial steps.

\subsection{Binomial expansions}
	\label{ssec:binomial}

The following two results will be used at various places throughout this paper.

\begin{proposition}[Central binomial coefficients, \cite{buric_elezovic,elezovic2015}]
	\label{prop:binomial}
	The central binomial coefficient has the asymptotic expansion
	\begin{equation}
		\label{exp-binom}
		\binom{2m}{m}\sim \frac{2^{2m}}{\sqrt{m\pi}}
		\sum_{k=0}^{\infty}\omega_k\,m^{-k},
	\end{equation}
	where $\omega_0=1$ and, for $k\ge1$,
	\begin{equation}
		\omega_k=\frac1k \sum_{j=1}^{\lfloor (k+1)/2\rfloor}
		\frac{(2^{-2j}-1)B_{2j}}{j}\,\omega_{k-2j+1},
	\end{equation}
	with $B_{2j}$ denoting the Bernoulli numbers. In particular
	$\omega_1=-\tfrac18$, $\omega_2=\tfrac1{128}$, $\omega_3=\tfrac5{1024}$.
\end{proposition}

	The expansion is used in the standard asymptotic sense as $m\to\infty$:
	truncation after the term $k=R$ leaves a relative error $O_R(m^{-R-1})$, with
	the implied constant depending on $R$.

\begin{proposition}[Shifted central binomial expansion, \cite{elezovic2015}]
	\label{prop:binomial-shift}
	With the shift $\alpha=\frac14$, the central binomial coefficient has the
	asymptotic expansion
	\begin{equation}
		\binom{2m}{m}\sim\frac{4^{m}}{\sqrt{\pi(m+\alpha)}}
		\left( \sum_{k=0}^{\infty}\tau_k(m+\alpha)^{-k}\right),
	\end{equation}
	where $\tau_0=1$ and, for $k\ge1$,
	\begin{equation}\label{alpha-14}
		\tau_k=\frac{1}{k}\sum_{j=1}^{\lfloor k/2\rfloor}2^{-4j-1}E_{2j}\,\tau_{k-2j}.
	\end{equation}
	Here $\{E_{2j}\}$ are the Euler numbers, $E_0=1$, $E_2=-1$, $E_4=5$,
	$E_6=-61,\dots$ (all odd-index Euler numbers vanish). Since $\tau_k$ with odd $k$
	refer only to odd indices and $\tau_1=0$, all odd-index coefficients vanish; in
	particular $\tau_2=-\tfrac1{64}$, $\tau_4=\tfrac{21}{8192}$,
	$\tau_6=-\tfrac{671}{524288}$.
\end{proposition}

	The recursion follows from
	$B_{2j+1}(\tfrac14)=-(2j+1)E_{2j}/2^{4j+2}$ in the logarithm of the shifted
	gamma quotient~\cite{elezovic2015}. As above, the expansion is used in the
	standard asymptotic sense as $m\to\infty$: truncation after the term $k=R$
	leaves a relative error $O_R\bigl((m+\alpha)^{-R-1}\bigr)$, the implied
	constant depending on the truncation order $R$. This is the form used in
	Theorem~\ref{thm:Ed-n2}.

\subsection{Definition and marginals}
	\label{ssec:mult-def}

	Let $X=(X_1,\dots,X_n)\sim\Mult(N;p_1,\dots,p_n)$, $\sum_i p_i=1$, count the
	observations falling into $n$ ordered bins in $N$ independent trials, with mass
	function
	\begin{equation}\label{eq:mult-pmf}
		\Pr(X_1=k_1,\dots,X_n=k_n)=\frac{N!}{k_1!\cdots k_n!}\,p_1^{k_1}\cdots p_n^{k_n},
		\qquad \sum_i k_i=N.
	\end{equation}
	A single bin is binomial, $X_i\sim\Bin(N,p_i)$; more generally the marginal of any
	subset of bins is again multinomial, the unused bins pooled into one remainder.

\begin{proposition}[$k$-bin marginal]\label{prop:k-marginal}
	For $S\subseteq\{1,\dots,n\}$ with $\sum_{i\in S}p_i<1$ (the strict inequality is intentional; the exhaustive case $\sum_{i\in S}p_i=1$ is not used), writing
	$X_S^c:=N-\sum_{i\in S}X_i$ for the pooled remainder,
	\[
		\bigl((X_i)_{i\in S},X_S^c\bigr)
		\sim
		\Mult\bigl(N;\,(p_i)_{i\in S},\,1-\textstyle\sum_{i\in S}p_i\bigr).
	\]
\end{proposition}

\begin{proof}
	Sum the mass function \eqref{eq:mult-pmf} over the counts outside $S$; by the
	multinomial theorem the inner sum equals $(1-\sum_{i\in S}p_i)^{m}/m!$ with
	$m=N-\sum_{i\in S}k_i$, which is the stated law.
\end{proof}

\subsection{Factorization of a two-bin marginal}
	\label{ssec:2bin-factorization}

	The summands of the total variation are entangled with the constraint
	$\sum_iX_i=N$. The following factorization removes that entanglement one edge at a
	time, reducing a single edge to a two-stage binomial experiment: first decide how
	many observations fall into the pair $\{j,j+1\}$, then split them between the two.
	Write $K_j:=X_j+X_{j+1}$ for the pair total.

\begin{proposition}[Factorization of the two-bin marginal]\label{prop:2bin-factor}
	If $p_j+p_{j+1}>0$ then $K_j\sim\Bin(N,p_j+p_{j+1})$, and conditionally
	\[
		X_{j+1}\mid K_j=k\;\sim\;\Bin(k,q_j),\qquad q_j:=\frac{p_{j+1}}{p_j+p_{j+1}}.
	\]
\end{proposition}

\begin{proof}
	$K_j$ counts the trials landing in $\{j,j+1\}$, so $K_j\sim\Bin(N,p_j+p_{j+1})$.
	If $p_j+p_{j+1}<1$, then given $K_j=k$, Proposition~\ref{prop:k-marginal} and
	\eqref{eq:mult-pmf} reduce $\Pr(X_{j+1}=m\mid K_j=k)$ to
	$\binom{k}{m}q_j^m(1-q_j)^{k-m}$, the law of $\Bin(k,q_j)$. If the pair is
	exhaustive ($p_j+p_{j+1}=1$, as at $n=2$ under the uniform law), then $K_j=N$
	deterministically and the conditional law is immediate from \eqref{eq:mult-pmf}
	itself: the $N$ trials split between the two bins as $\Bin(N,q_j)$.
\end{proof}

	Hence, conditionally on $K_j=k$, the edge $D_j=2X_{j+1}-k$ is a shifted binomial;
	under the \emph{uniform} law ($p_j=p_{j+1}$) one has $q_j=\tfrac12$, so $D_j\mid
	K_j=k$ is the position of a symmetric $\pm1$ walk of $k$ steps --- the
	identification on which every computation below rests.

\begin{definition}[Edge variables]\label{def:Dj}
	The \emph{edge variables} are $D_j:=X_{j+1}-X_j$, $j=1,\dots,n-1$, together with
	the wrap edge $D_n:=X_1-X_n$ (the convention $X_{n+1}:=X_1$), used by the cyclic
	statistic. Their centered
	versions are $\tilde D_j:=D_j-\E D_j=D_j-N(p_{j+1}-p_j)$. Under the uniform law
	$\E D_j=0$, so $\tilde D_j=D_j$.
\end{definition}

\section{Expectation of DTV under the uniform distribution}
	\label{sec:E-uniform}

	Throughout the paper the sample size satisfies $N\ge1$; the empty sample $N=0$ carries no information and is excluded, so the statements below omit this standing hypothesis.

	The uniform null is where the theory is cleanest: the symmetry of the uniform
	multinomial makes every edge interchangeable, so the mean of $d$ reduces to that
	of a single edge. That mean is the yardstick against which an observed total
	variation is judged significant, and it is worth knowing exactly, not merely up
	to a constant.

	We obtain it in two stages. First, \S\ref{sec:Ed-exact} gives an \emph{exact}
	closed form in central trinomial coefficients, valid for every $n\ge2$ and
	specializing to one-term formulas at $n=2,3,4$. Then \S\ref{sec:Ed-asymptotic}
	develops the large-$N$ \emph{asymptotics} by mixture-conditioning --- less
	explicit but portable, and the route that survives the passage to non-uniform
	reference laws in the later parts.

	After normalization, the cyclic mean tends to $2/\sqrt\pi$ \emph{regardless of the
	number of bins}: a quantity that looks as though it ought to depend on the whole
	shape of the problem collapses, in the limit, to an $n$-independent constant ---
	the first of several such collapses in the moment theory of DTV. This value is no
	accident: $\Var D_1=2N/n$ exactly, so $\E|D_1|\sim2\sqrt{N/(n\pi)}$ by the central
	limit theorem, and summing the $n$ edges and dividing by $\sqrt{Nn}$ returns
	$2/\sqrt\pi$; the content below is the exact value and its complete expansion.


\subsection{Setup: DTV functionals under the uniform}
	\label{ssec:dtv-uniform-setup}

	From this point on both distributions are uniform: the sample is drawn from
	$X = (X_1, \ldots, X_n)\sim\Mult(N;\mathbf u)$, with
	$\mathbf u=(1/n,\ldots,1/n)$, and the reference distribution in the statistic is
	the same $\mathbf u$. Thus this paper treats the $(\mathbf u,\mathbf u)$ case of
	the general statistic $d_{\mathbf r}$. Since the uniform centering cancels on
	each edge, the \emph{linear} DTV is
	\begin{equation}\label{eq:d-lin}
		d := \sum_{i=1}^{n-1} |X_{i+1} - X_i| = \sum_{j=1}^{n-1} |D_j|,
	\end{equation}
	and the \emph{cyclic} (circular) DTV is
	\begin{equation}\label{eq:d-cyc}
		d^\circ := \sum_{i=1}^{n} |X_{i+1} - X_i|, \qquad X_{n+1} := X_1.
	\end{equation}
	(For $n=2$ the cycle traverses the single edge twice, $d^\circ=2|X_2-X_1|$; the
	formulas below cover this case consistently.)
	The edge variables $D_j = X_{j+1} - X_j$ are defined in
	Definition~\ref{def:Dj}. Under the uniform distribution $\E[D_j] = 0$, so centering
	is trivial: $\tilde D_j \equiv D_j$.

	Under the uniform multinomial the bins are exchangeable, so all edges $|D_j|$
	have the same expectation:
	\begin{equation}\label{eq:Ed-Edc-symmetry}
		\E(d^\circ) = n \cdot \E|X_2 - X_1|, \qquad \E(d) = (n-1) \cdot \E|X_2 - X_1|,
	\end{equation}
	that is,
	\begin{equation}\label{eq:Edc-Ed-ratio}
		\E(d^\circ) = \frac{n}{n-1}\,\E(d).
	\end{equation}
	Hence it suffices to compute $\E|D_1| = \E|X_2 - X_1|$. The constant ratio
	\eqref{eq:Edc-Ed-ratio} is special to the uniform case, where exchangeability makes
	all edges identically distributed; for a general reference law $\mathbf r$ the cyclic and linear
	statistics differ by the single wrap-edge term, not by a constant factor.

\subsection{Auxiliary lemma: expectation of a binomial with $p = 1/2$}
	\label{ssec:E-Sk-bin}

	One small identity underlies everything that follows, so we isolate it first.
	A fair coin tossed $m$ times leaves the gambler, on average, a definite distance
	from the break-even point, and that distance --- the mean absolute displacement of
	a symmetric walk --- is exactly a central binomial coefficient in disguise. The
	identity is classical: it is the symmetric case $p=\tfrac12$ of de~Moivre's
	mean-absolute-deviation formula for the binomial, whose history is surveyed by
	Diaconis and Zabell~\cite{diaconis_zabell}; we include the short proof for
	completeness. The coefficient is
	a recurring one: we shall meet $\binom{m-1}{\lfloor(m-1)/2\rfloor}$ and its relatives
	at every turn, because the uniform multinomial dissolves, edge by edge, into
	precisely such walks.

\begin{lemma}[Position of a symmetric $\pm 1$ random walk]
	\label{lem:E-bin-walk}
	Let $Y \sim \Bin(m, 1/2)$, $m \ge 1$. Then
	\begin{equation}\label{eq:E-bin-half}
		\E|m - 2Y| = \frac{m}{2^{m-1}}\binom{m-1}{\lfloor (m-1)/2 \rfloor}.
	\end{equation}
\end{lemma}

\begin{proof}
	We split according to the parity of $m$.

	\textit{Case $m = 2r$ even.} By the symmetry $|2r - 2k| = 2|r-k|$ and
	$\binom{2r}{k} = \binom{2r}{2r-k}$:
	\[
		\E|m-2Y| = 2^{-2r}\sum_{k=0}^{2r}|2r-2k|\binom{2r}{k}
		= 4 \cdot 2^{-2r}\sum_{k=0}^{r-1}(r-k)\binom{2r}{k}.
	\]
	The identity
	\begin{equation}\label{eq:bin-identity-even}
		\sum_{k=0}^{r-1}(r-k)\binom{2r}{k} = r\binom{2r-1}{r-1}
	\end{equation}
	follows by summing the telescoping identity
	\[
		(r-k)\binom{2r}{k}
		=r\left[\binom{2r-1}{k}-\binom{2r-1}{k-1}\right],
	\]
	with $\binom{2r-1}{-1}=0$. Hence
	\[
		\E|m-2Y| = \frac{4r\binom{2r-1}{r-1}}{2^{2r}}
		= \frac{2r}{2^{2r-1}}\binom{2r-1}{r-1}
		= \frac{m}{2^{m-1}}\binom{m-1}{(m-2)/2},
	\]
	and $\lfloor (m-1)/2\rfloor = r-1 = (m-2)/2$.

	\textit{Case $m = 2r+1$ odd.} By the substitution $k \mapsto m-k$:
	\[
		\E|m-2Y| = 2 \cdot 2^{-m}\sum_{k=0}^{r}(m-2k)\binom{m}{k}
		= 2^{1-m}\Bigl[m\sum_{k=0}^{r}\binom{m}{k} - 2\sum_{k=0}^{r}k\binom{m}{k}\Bigr].
	\]
	The first sum is $2^{m-1}$ (by symmetry and the parity of $m$). The second, using
	$k\binom{m}{k} = m\binom{m-1}{k-1}$:
	\[
		\sum_{k=0}^{r}k\binom{m}{k}
		= m\sum_{j=0}^{r-1}\binom{m-1}{j}
		= \tfrac{1}{2}m\bigl(2^{m-1} - \binom{m-1}{r}\bigr).
	\]
	Indeed, $m-1=2r$ is even, so the terms of its binomial expansion pair
	symmetrically around the unpaired central term $\binom{2r}{r}$:
	\[
		2\sum_{j=0}^{r-1}\binom{2r}{j}+\binom{2r}{r}=2^{2r}.
	\]
	This gives the last equality above and makes explicit why the central binomial
	coefficient survives.
	Hence
	\[
		\E|m-2Y| = 2^{1-m}\bigl[m\cdot 2^{m-1} - m\cdot 2^{m-1} + m\binom{m-1}{r}\bigr]
		= \frac{m}{2^{m-1}}\binom{m-1}{r},
	\]
	and $r = \lfloor (m-1)/2\rfloor$.
\end{proof}

\subsection{Exact formula for $\E(d)$: single sum}
	\label{ssec:Ed-exact-sum}

\begin{theorem}[Closed sum for the expectation]\label{thm:Ed-sum}
	Under the uniform multinomial distribution with sample size $N$ and $n\ge2$ bins:
	\begin{equation}\label{eq:Ed-sum}
		\E(d) = \frac{2N(n-1)}{n^N}
		\sum_{j=0}^{N-1}\binom{N-1}{j}\binom{j}{\lfloor j/2\rfloor}(n-2)^{N-1-j}.
	\end{equation}
	For $n>2$, equivalently:
	\begin{equation}\label{eq:Ed-sum-alt}
		\E(d) = 2(n-1)\Bigl(\frac{n-2}{n}\Bigr)^{N}\frac{N}{n-2}
		\sum_{j=0}^{N-1}\binom{N-1}{j}\binom{j}{\lfloor j/2\rfloor}(n-2)^{-j}.
	\end{equation}
	For $n=2$, the first formula is read with $0^0=1$; only the term
	$j=N-1$ survives.
\end{theorem}

\begin{proof}
	By symmetry \eqref{eq:Ed-Edc-symmetry}, $\E(d) = (n-1)\E|X_2 - X_1|$, and for
	$n\ge3$ the pair
	$(X_1, X_2)$ has the trinomial marginal $\Mult(N; 1/n, 1/n, (n-2)/n)$
	(Proposition~\ref{prop:k-marginal}); at $n=2$ the ``trinomial'' degenerates to the
	binomial pair itself, covered by the exhaustive case of
	Proposition~\ref{prop:2bin-factor}. Proposition~\ref{prop:2bin-factor} gives
	\[
		M := X_1 + X_2 \sim \Bin(N, 2/n), \quad X_2 \mid M = m \sim \Bin(m, 1/2),
		\quad X_1 = M - X_2.
	\]
	Hence $|X_2 - X_1| = |2X_2 - M|$ and by conditioning
	\[
		\E|X_2-X_1| = \sum_{m=0}^{N}\binom{N}{m}\Bigl(\frac{2}{n}\Bigr)^m
		\Bigl(\frac{n-2}{n}\Bigr)^{N-m}\!\!\E|m-2X_2|.
	\]
	The $m=0$ term contributes $0$; for $m\ge1$, Lemma~\ref{lem:E-bin-walk} gives:
	\[
		\E|X_2-X_1| = \sum_{m=1}^{N}\binom{N}{m}\Bigl(\frac{2}{n}\Bigr)^m
		\Bigl(\frac{n-2}{n}\Bigr)^{N-m}\!\frac{m}{2^{m-1}}\binom{m-1}{\lfloor (m-1)/2\rfloor}.
	\]
	Using $m\binom{N}{m} = N\binom{N-1}{m-1}$ and the index $j := m-1$:
	the remaining powers simplify as
	\[
		\Bigl(\frac2n\Bigr)^m\frac{1}{2^{m-1}}
		\Bigl(\frac{n-2}{n}\Bigr)^{N-m}
		=\frac{2(n-2)^{N-m}}{n^N}.
	\]
	Therefore
	\[
		\E|X_2-X_1| = \frac{2N}{n^N}
		\sum_{j=0}^{N-1}\binom{N-1}{j}\binom{j}{\lfloor j/2\rfloor}(n-2)^{N-1-j}.
	\]
	Multiplying by $n-1$ gives \eqref{eq:Ed-sum}. The form \eqref{eq:Ed-sum-alt}
	follows by extracting $(n-2)^{N-1}$ from the sum; it is numerically stabler for
	large $n$ (it factors out the dominant power $(n-2)^{N-1}$) and is the form from which the $n=3$
	specialization of \S\ref{ssec:small-n-exact} is read off.
\end{proof}

\begin{remark}[Prior appearance in TVOR]\label{rem:tvor-credit}
	The expectation of $d$ under the uniform multinomial was first studied
	in~\cite{banic_elezovic_tvor}, which gave the symmetry reduction
	\eqref{eq:Ed-Edc-symmetry}, the cyclic ratio~\eqref{eq:Edc-Ed-ratio}, the exact
	value at $n=2$, an exact \emph{double}-sum formula for general $n$ [Thm.~2], and
	conjectured the leading asymptote $2(n-1)\sqrt{N/(n\pi)}$ [Hyp.~1]. The present
	paper reduces that double sum to the single sum~\eqref{eq:Ed-sum}, gives a closed
	form valid for every $n\ge2$, and proves and completes the asymptotic expansion.
\end{remark}

\begin{remark}[A Gini-type mean difference]\label{rem:gini}
	At bottom, Theorem~\ref{thm:Ed-sum} evaluates the mean difference
	$\E|X_i-X_j|$ of two cells of one multinomial --- a Gini-type quantity with the
	remaining $n-2$ cells mixed out. Beyond the double-sum evaluation of~\cite[Thm.~2]{banic_elezovic_tvor}, we are not aware of a closed-form treatment of
	this exact quantity in the mean-difference literature, whose exact computations
	concern i.i.d.\ samples or independent binomials; the adjacent classical result
	is de~Moivre's identity recalled in Lemma~\ref{lem:E-bin-walk}.
\end{remark}

\section{Exact closed form via central trinomial coefficients}
	\label{sec:Ed-exact}

	The single sum \eqref{eq:Ed-sum} is a purely combinatorial object --- binomial
	coefficients and central binomials, nothing continuous in sight --- and it closes
	completely. The edge variable $X_2-X_1$ is the position after $N$ steps of a
	\emph{lazy} symmetric walk, with steps $+1$ and $-1$ each of probability $1/n$ and
	a hold of probability $(n-2)/n$; its central weights are the generalized central
	trinomial coefficients evaluated at $b=n-2$. A single generating function turns
	this observation into a closed form for $\E(d)$ valid for every $n\ge2$, the
	bedrock of the analytic study of the expectation.

\subsection{The closed form}
	\label{ssec:trinomial}

	The \emph{generalized central trinomial coefficient} is
	\begin{equation}\label{eq:Tk-def}
		T_k(b):=[t^k]\,(1+bt+t^2)^k
		=\sum_{i\ge0}\binom{k}{2i}\binom{2i}{i}b^{k-2i},
	\end{equation}
	with the diagonal generating function
	\begin{equation}\label{eq:Tk-gf}
		\sum_{k\ge0}T_k(b)\,z^k=\bigl(1-2bz+(b^2-4)z^2\bigr)^{-1/2};
	\end{equation}
	thus $T_k(0)=\binom{k}{k/2}$ for even $k$, $T_k(1)$ is the central trinomial
	sequence, and $T_k(2)=\binom{2k}{k}$. Throughout write $T_k:=T_k(n-2)$.

\begin{theorem}[Closed form for $\E(d)$]\label{thm:Ed-closed}
	For every $n\ge2$ and $N\ge1$,
	\begin{equation}\label{eq:Ed-closed}
		\E(d) = \frac{N(n-1)}{n^N}\bigl[T_N-(n-4)\,T_{N-1}\bigr].
	\end{equation}
	Equivalently for the cyclic DTV:
	\begin{equation}\label{eq:Edc-closed}
		\E\!\left(\frac{d^\circ}{\sqrt{Nn}}\right)
		= \frac{\sqrt{Nn}}{n^N}\bigl[T_N-(n-4)\,T_{N-1}\bigr].
	\end{equation}
\end{theorem}

\begin{proof}
	From Theorem~\ref{thm:Ed-sum},
	$\E(d) = \tfrac{2N(n-1)}{n^N} S_N$ where
	\[
		S_N := \sum_{j=0}^{N-1}\binom{N-1}{j}\binom{j}{\lfloor j/2\rfloor}(n-2)^{N-1-j},
		\quad
		\beta_j := \binom{j}{\lfloor j/2\rfloor}.
	\]
	The goal is to identify $S_N$ through the trinomial generating
	function~\eqref{eq:Tk-gf}, in four steps.

	\medskip
	\textit{Step 1. Generating function for $\beta_j$.} Split by parity,
	using $\binom{2k+1}{k} = \tfrac{1}{2}\binom{2k+2}{k+1}$:
	\begin{gather*}
		\sum_{k\ge 0}\binom{2k}{k}w^{2k} = (1-4w^2)^{-1/2},
		\\
		\sum_{k\ge 0}\binom{2k+1}{k}w^{2k+1} = \frac{1}{2w}\bigl[(1-4w^2)^{-1/2} - 1\bigr].
	\end{gather*}
	Their sum:
	\[
		A(w) := \sum_{j \ge 0} \beta_j w^j
		= (1-4w^2)^{-1/2}\Bigl(1 + \frac{1}{2w}\Bigr) - \frac{1}{2w}
		= \frac{1}{2w}\Bigl[\sqrt{\frac{1+2w}{1-2w}} - 1\Bigr].
	\]

	\medskip
	\textit{Step 2. Generating function for $S_N$.} We compute
	\[
		F(z) := \sum_{N \ge 1} S_N\,z^N.
	\]
	By the change of index $M = N-1$ and interchanging the order of summation:
	\[
		F(z) = z\sum_{j \ge 0}\beta_j\sum_{M\ge j}\binom{M}{j}z^M (n-2)^{M-j}
		= z\sum_{j \ge 0} \beta_j \cdot \frac{z^j}{(1-(n-2)z)^{j+1}},
	\]
	using the standard identity $\sum_{M \ge j}\binom{M}{j}u^{M-j} = (1-u)^{-(j+1)}$ with
	$u = (n-2)z$. Setting $w := z/(1-(n-2)z)$:
	\[
		F(z) = \frac{z}{1-(n-2)z}\cdot A(w)
		= \frac{1}{2}\Bigl[\sqrt{\frac{1+2w}{1-2w}} - 1\Bigr].
	\]
	We compute
	\[
		1+2w = \frac{1-(n-4)z}{1-(n-2)z}, \quad
		1-2w = \frac{1-nz}{1-(n-2)z}.
	\]
	Hence
	\begin{equation}\label{eq:Fz-gen}
		F(z) = \frac{1}{2}\Bigl[\sqrt{\frac{1-(n-4)z}{1-nz}} - 1\Bigr].
	\end{equation}
	All manipulations in this step may be read as identities of formal power
	series. In particular, every coefficient of $z^N$ receives contributions from
	only finitely many indices $j$, so the interchange does not require an
	independent analytic convergence argument.

	\medskip
	\textit{Step 3. The trinomial generating function.} With $b=n-2$ the discriminant
	of the quadratic in \eqref{eq:Tk-gf} is $b^2-4=n(n-4)$, so
	\[
		1 - 2(n-2)z + n(n-4)z^2 = (1-nz)(1-(n-4)z)
	\]
	(roots $1/n$ and $1/(n-4)$), and \eqref{eq:Tk-gf} reads
	\[
		\sum_{N \ge 0} T_N\,z^N = \frac{1}{\sqrt{(1-nz)(1-(n-4)z)}}.
	\]

	\medskip
	\textit{Step 4. Combining.} Multiplying the Step~3 identity by $(1-(n-4)z)$,
	\[
		(1-(n-4)z)\sum_{N\ge 0} T_N\,z^N
		= \sqrt{\frac{1-(n-4)z}{1-nz}} = 2F(z)+1
	\]
	by \eqref{eq:Fz-gen}. The coefficient of $z^N$ ($N\ge1$) on the left is
	$T_N-(n-4)T_{N-1}$, and on the right it is $2S_N$; hence
	\[
		S_N = \tfrac12\bigl[T_N-(n-4)T_{N-1}\bigr],
		\qquad
		\E(d) = \frac{2N(n-1)}{n^N}S_N
		= \frac{N(n-1)}{n^N}\bigl[T_N-(n-4)T_{N-1}\bigr],
	\]
	which is \eqref{eq:Ed-closed}. The cyclic form \eqref{eq:Edc-closed} follows from
	$\E(d^\circ)=\tfrac{n}{n-1}\E(d)$ \eqref{eq:Edc-Ed-ratio} on dividing by
	$\sqrt{Nn}$, since $nN/\sqrt{Nn}=\sqrt{Nn}$. As \eqref{eq:Tk-gf} is an identity of
	formal power series for every real $b$, \eqref{eq:Ed-closed} holds for all $n\ge2$
	with no exceptional values.
\end{proof}

\begin{remark}[Small cases and the classical connection]\label{rem:trinomial}
	The three bins with a one-term evaluation are read from \eqref{eq:Ed-closed}: at
	$n=4$ the factor $n-4$ annihilates the second term and $T_N(2)=\binom{2N}{N}$,
	giving $\E(d)|_{n=4}=3N\,4^{-N}\binom{2N}{N}$ (Theorem~\ref{thm:Ed-n4}); at $n=2$
	($b=0$) only one parity of $T_k(0)$ survives (Theorem~\ref{thm:Ed-n2}); at $n=3$
	($b=1$) the combination $\tfrac12[T_N(1)+T_{N-1}(1)]$ is the directed-animal
	sequence (Proposition~\ref{prop:n3}).
	Through \eqref{eq:Tk-gf} the coefficients $T_k(n-2)$ are classical values of the
	Legendre polynomials, $T_k(b)=(b^2-4)^{k/2}P_k\bigl(b/\sqrt{b^2-4}\bigr)$
	(Good~\cite{good1958}; Noe~\cite{noe2006}); the device of recognizing
	$[(1-az)(1-bz)]^{-1/2}$ as a Legendre generating function is classical, its
	best-known instance being the central Delannoy numbers $D_N=P_N(3)$
	(Comtet~\cite{comtet}; Banderier and Schwer~\cite{banderier_schwer}). That
	connection, first pointed out to the authors by L.~Mihokovi\'c, affords the
	complete large-$N$ asymptotics of $T_k(b)$ through the Laplace--Heine formula ---
	classical, and carried to all orders by Wagner~\cite{wagner2012} (proving a
	conjecture of Z.-W.~Sun). The present paper instead derives the asymptotics of
	$\E(d)$ self-containedly by mixture-conditioning (\S\ref{sec:Ed-asymptotic}),
	needing only the elementary generating function \eqref{eq:Tk-gf}.
\end{remark}

	For large $N$ the bracket $T_N-(n-4)T_{N-1}$ in \eqref{eq:Ed-closed} is a
	difference of two exponentially large integers; for a practical computation one
	uses the exact three-term recurrence of Lemma~\ref{lem:exact-W-recurrence}
	(Remark~\ref{rem:exact-tabulation}) or the asymptotic expansion of
	\S\ref{sec:Ed-asymptotic}.

\subsection{Special cases: exact formulas for small $n$}
	\label{ssec:small-n-exact}

	Theorem~\ref{thm:Ed-closed} specializes to one-line closed forms at $n=2,3,4$,
	recorded here before the asymptotic analysis; all three follow by evaluating the
	central trinomial coefficient $T_k(n-2)$ at $b=n-2\in\{0,1,2\}$.

\paragraph{The case $n=2$}
	Since $(n-2)^{N-1-\ell}=0$ for $\ell<N-1$, the closed sum collapses to its top
	term, $\E(d)|_{n=2}=N\binom{N-1}{\lfloor(N-1)/2\rfloor}2^{-(N-1)}$, a single
	central binomial coefficient; its complete asymptotic expansion is given in
	Theorem~\ref{thm:Ed-n2}.

\paragraph{The case $n=3$}
	\label{ssec:n3-case}

	At $n=3$ ($b=1$) the closed form \eqref{eq:Ed-closed}, with $n-4=-1$, reads
	$S_N|_{n=3}=\tfrac12[T_N(1)+T_{N-1}(1)]=:Q_N$, the half-sum of two consecutive
	central trinomial coefficients; the asymptotics are supplied in a common form for
	every fixed $n\ge3$ by the mixture-conditioning analysis of
	\S\ref{sec:Ed-asymptotic}. The generating function from Step~2 of the proof of
	Theorem~\ref{thm:Ed-closed},
	\begin{equation}\label{eq:Fz-n3}
		F(z)|_{n=3} = \frac{1}{2}\left[\sqrt{\frac{1+z}{1-3z}} - 1\right],
	\end{equation}
	is regular for $|z| < 1/3$, and the coefficients $Q_N$ are the \emph{binomial
	transform}~\eqref{eq:SN-n3} of the central weights
	$\binom{j}{\lfloor j/2\rfloor}$. The finite sum~\eqref{eq:SN-n3}, the generating
	function~\eqref{eq:Fz-n3}, and the $P$-recurrence~\eqref{eq:SN3-recurrence} of the
	following proposition together give a complete D-finite characterization.

\begin{proposition}[Exact characterization at $n=3$]\label{prop:n3}
	For the uniform three-bin law,
	\begin{equation}\label{eq:Ed-n3}
		\E(d)|_{n=3} = \frac{4N}{3^N}\,Q_N,
		\qquad
		\E(d^\circ)|_{n=3} = \frac{6N}{3^N}\,Q_N,
	\end{equation}
	where
	\begin{equation}\label{eq:SN-n3}
		Q_N := \sum_{j = 0}^{N-1}\binom{N-1}{j}\binom{j}{\lfloor j/2\rfloor}.
	\end{equation}
	The function $f(z):=1+2\sum_{N\ge1}Q_Nz^N=\sqrt{(1+z)/(1-3z)}$ is algebraic and
	satisfies the first-order linear differential equation
	\[
		(1 - 2z - 3z^2)\,f'(z) = 2\,f(z);
	\]
	consequently the sequence $(Q_N)$ is D-finite and obeys the second-order (three-term) $P$-recurrence
	\begin{equation}\label{eq:SN3-recurrence}
		(N+1)\,Q_{N+1} = (2N+2)\,Q_N + 3(N-1)\,Q_{N-1}
		\qquad (N \ge 1),\quad Q_1 = 1.
	\end{equation}
\end{proposition}

\begin{proof}
	Substituting $n=3$ into the closed sum (Theorem~\ref{thm:Ed-sum}) --- where
	$(n-2)^{N-1-j}=1$ and $2N(n-1)/n^N=4N/3^N$ --- gives \eqref{eq:Ed-n3} with
	\eqref{eq:SN-n3}; this is the specialization of the closed sum, not a separate
	result (in the notation of \S\ref{ssec:trinomial}, $Q_N=S_N|_{n=3}$, just
	as $T_M=S_{M+1}$ in Lemma~\ref{lem:exact-W-recurrence} below). By
	\eqref{eq:Fz-n3} and $f=1+2F|_{n=3}$, $f(z)=\sqrt{(1+z)/(1-3z)}$. Logarithmic
	differentiation gives
	$f'/f = \tfrac1{2(1+z)} + \tfrac3{2(1-3z)} = \tfrac2{(1+z)(1-3z)}$, and
	$(1+z)(1-3z) = 1 - 2z - 3z^2$, which is the differential equation. Writing
	$f(z) = 1 + 2\sum_{N\ge1} Q_N z^N$ and comparing coefficients of $z^N$: the
	left side contributes $2(N+1)Q_{N+1}-4N Q_N-6(N-1)Q_{N-1}$, the right side
	$4Q_N$; division by $2$ gives \eqref{eq:SN3-recurrence}. A single seed
	suffices: at $N=1$ the $Q_0$-term carries the factor $3(N-1)=0$, so $Q_2=2$
	follows from the recurrence itself.
\end{proof}

\begin{remark}\label{rem:n3-oeis}
	$Q_1, Q_2, \ldots, Q_7 = 1, 2, 5, 13, 35, 96, 267$: these are the
	\emph{directed animals of size $N$} (equivalently, base-$3$ $N$-digit numbers
	with digit sum $N$), $Q_N = \mathrm{A005773}(N)$ in the OEIS~\cite{oeis}
	($N \ge 1$), and the recurrence \eqref{eq:SN3-recurrence} reproduces them
	exactly. The functional relation $f(z)^2(1-3z)=1+z$ also yields a quadratic
	(implicit) recurrence, superseded by the linear \eqref{eq:SN3-recurrence}.
	The sequence is D-finite~\cite{flajolet_sedgewick} but \emph{not} C-finite: a
	sequence satisfying a linear recurrence with \emph{constant} coefficients has a
	\emph{rational} ordinary generating function, whereas $f$ is algebraic and
	irrational.
\end{remark}

\paragraph{The case $n=4$}
	\label{ssec:n4-case}

	At $n=4$ ($b=2$) the factor $n-4$ in \eqref{eq:Ed-closed} annihilates the second
	term, and $T_N(2)=\binom{2N}{N}$ is a single central binomial coefficient --- the
	one bin count at which $\E(d)$ reduces to a single term. Equivalently, from
	Step~2 of the proof of Theorem~\ref{thm:Ed-closed},
	\[
		F(z)|_{n=4} = \tfrac{1}{2}\bigl[(1 - 4z)^{-1/2} - 1\bigr]
		= \tfrac{1}{2}\sum_{N \ge 1}\binom{2N}{N}z^N,
	\]
	using $(1 - 4z)^{-1/2} = \sum_{N \ge 0}\binom{2N}{N}z^N$, so that
	\begin{equation}\label{eq:SN-n4}
		S_N|_{n=4} = \tfrac{1}{2}\binom{2N}{N}.
	\end{equation}

\begin{theorem}[Closed formula for $\E(d)$ at $n = 4$]\label{thm:Ed-n4}
	For the multinomial uniform distribution with $n = 4$ bins and sample size $N$:
	\begin{equation}\label{eq:Ed-n4}
		\E(d)|_{n=4} = \frac{3N}{4^N}\binom{2N}{N},
		\qquad
		\E(d^\circ)|_{n=4} = \frac{N}{4^{N-1}}\binom{2N}{N}.
	\end{equation}
\end{theorem}

\begin{proof}
	From Theorem~\ref{thm:Ed-sum}, $\E(d) = \tfrac{2N(n-1)}{n^N}S_N$. At $n = 4$:
	\[
		\E(d)|_{n=4} = \frac{6N}{4^N}\cdot\tfrac{1}{2}\binom{2N}{N} = \frac{3N\binom{2N}{N}}{4^N}.
	\]
	The cyclic version from $\E(d^\circ) = \tfrac{n}{n-1}\E(d) = \tfrac{4}{3}\E(d)$.
\end{proof}

	The boundary case fits the general pattern: Stirling on~\eqref{eq:Ed-n4} gives
	$\E(d^\circ/\sqrt{Nn})|_{n=4}=\tfrac{2}{\sqrt\pi}-\tfrac{1}{4\sqrt\pi\,N}+O(N^{-2})$,
	so the coefficient $-(n-2)/(8\sqrt\pi\,N)$ of the expansion of
	\S\ref{sec:Ed-asymptotic} (Theorem~\ref{thm:main-expansion},
	Corollary~\ref{cor:coefficients}) continues correctly to $n=4$.

	\emph{Special status of $n=4$.} With $n=4$ the statistic $d$ has three summands,
	the smallest case admitting a \emph{disjoint} pair of adjacent differences
	($|X_2-X_1|$ and $|X_4-X_3|$, in fact the only such pair); this makes $n=4$ the
	natural starting point for the variance analysis of Part~2.

\section{Asymptotics via mixture-conditioning}
	\label{sec:Ed-asymptotic}

	Having the exact closed form, we turn to the large-$N$ behavior, which it does not
	display directly: the bracket $T_N-(n-4)T_{N-1}$ is a difference of exponentially
	large terms. Mixture-conditioning supplies the asymptotics --- it conditions on
	the sum $K = X_1 + X_2$, turns the difference into a simple random walk, and
	delivers the full expansion of $\E|D_1|$. The method is also portable: it is
	generalized to the nonuniform case in Part~3, on the expectation under a general
	reference law, whereas the closed form of \S\ref{sec:Ed-exact} is tied to the
	uniform symmetry.

	Under the uniform multinomial,
	$K = X_1 + X_2 \sim \Bin(N, 2/n)$ (Proposition~\ref{prop:2bin-factor})
	with $\mu_K = 2N/n$, $\Var(K) = (2N/n)(1-2/n)$. Conditionally $X_2 \mid K=k \sim \Bin(k, 1/2)$
	(since $q = 1/2$ for the uniform). So
	\begin{equation}\label{eq:D1-as-SK}
		D_1 = X_2 - X_1 = 2X_2 - K =: S_K,
	\end{equation}
	the position of a symmetric $\pm 1$ random walk with $K$ steps.
	Put $W_N:=\E|D_1|=\E|S_K|$; throughout this section the number of bins $n$ is
	fixed and suppressed from the notation $W_N$, while the expansion coefficients
	$e_k(n)$ display it.

\subsection{Auxiliary lemmas}
	\label{ssec:lemmas}

\begin{lemma}[Mixture-conditioning and existence of the expansion]
	\label{lem:mixture-existence}
	For every fixed $n\ge3$, $W_N$ has a complete asymptotic expansion
	\[
		W_N\sim
		2\sqrt{\frac{N}{n\pi}}
		\left(1+\sum_{k\ge1}\frac{e_k(n)}{N^k}\right).
	\]
\end{lemma}

\begin{proof}
	Set $\theta=2/n$ and condition on
	\[
		K:=X_1+X_2\sim\Bin(N,\theta).
	\]
	Conditionally on $K=k$, one has
	\[
		X_2\mid K=k\sim\Bin(k,1/2),\qquad
		D_1=2X_2-k=:S_k,
	\]
	where $S_k$ is the position of the symmetric simple walk after $k$
	steps. Thus
	\[
		W_N=\E g(K),\qquad g(k):=\E|S_k|.
	\]

	The exact formula for the walk expectation is
	\[
		g(k)=\frac{k}{2^{k-1}}\,\binom{k-1}{\lfloor(k-1)/2\rfloor}
		\qquad(k\ge1),\qquad g(0)=0.
	\]
	We spell out the parity split, since this is the point where the lattice
	oscillation enters. If $k=2s$, then
	\[
		g(2s)=\frac{2s}{4^s}\binom{2s}{s},
	\]
	because \(\binom{2s-1}{s-1}=\frac12\binom{2s}{s}\). Proposition~\ref{prop:binomial}
	therefore gives
	\[
		g(2s)
		=\sqrt{\frac{4s}{\pi}}
		\left(1-\frac{1}{8s}+\frac{1}{128s^2}+O(s^{-3})\right)
		=\sqrt{\frac{2k}{\pi}}
		\left(1-\frac{1}{4k}+\frac{1}{32k^2}+O(k^{-3})\right).
	\]
	If $k=2s+1$, then
	\[
		g(2s+1)=\frac{2s+1}{4^s}\binom{2s}{s}.
	\]
	Using again Proposition~\ref{prop:binomial}, now with \(s=(k-1)/2\), gives
	\[
		g(k)
		=\sqrt{\frac{2k}{\pi}}
		\left(1+\frac{1}{4k}+\frac{1}{32k^2}+O(k^{-3})\right)
		\qquad(k\ \hbox{odd}).
	\]
	The same calculation to any prescribed order produces two asymptotic series,
	one for even \(k\) and one for odd \(k\). Taking their half-sum and half-difference
	repackages them as a smooth part plus a period-$2$ part: for every fixed \(R\),
	\[
		g(k)=
		\sqrt{\frac{2k}{\pi}}
		\left(\sum_{h=0}^{R}b_h k^{-h}
		+(-1)^k\sum_{h=1}^{R}c_h k^{-h}\right)
		+O(k^{1/2-R-1}),
	\]
	with, in particular, \(b_0=1\), \(b_1=0\), \(b_2=1/32\), and \(c_1=-1/4\).
	The remaining constants \(b_h,c_h\) are determined by the central-binomial
	coefficients; their explicit form is not needed here. The precise meaning of
	the remainder matters: Proposition~\ref{prop:binomial} is an asymptotic
	expansion as $k\to\infty$, so for each fixed $R$ there are constants $k_R$ and
	$C_R$ such that the display holds with absolute remainder at most
	$C_R\,k^{1/2-R-1}$ for all $k\ge k_R$. We shall use the expansion only on the
	event $K\ge\mu/2=N\theta/2$, where $k\asymp N$; for all sufficiently large $N$
	the estimate is therefore uniform on the whole region in which it is applied.

	The oscillating part is exponentially small after averaging. Indeed, for
	$0<\theta<1$ put $q=\max\{1-\theta,|1-2\theta|\}<1$. Then for every
	fixed $\alpha>0$,
	\[
		\E[(-1)^K K^{-\alpha};K\ge1]=O(Nq^N)=O(e^{-cN})
	\]
	for some $c>0$. To prove this,
	use the Laplace representation
	\[
		k^{-\alpha}=\Gamma(\alpha)^{-1}\int_0^\infty t^{\alpha-1}e^{-kt}\,dt.
	\]
	Then
	\[
		\E[(-1)^K K^{-\alpha};K\ge1]
		=\frac1{\Gamma(\alpha)}\int_0^\infty t^{\alpha-1}
		\bigl[(1-\theta-\theta e^{-t})^N-(1-\theta)^N\bigr]\,dt.
	\]
	For $0\le y=e^{-t}\le1$ the mean-value theorem gives
	\[
		|(1-\theta-\theta y)^N-(1-\theta)^N|\le N\theta y\,q^{N-1},
	\]
	so the integral is at most
	\[
		N\theta q^{N-1}\Gamma(\alpha)^{-1}\int_0^\infty t^{\alpha-1}e^{-t}\,dt=N\theta q^{N-1}
	\]
	in absolute value. Since $q<1$ is fixed, $Nq^N=O(e^{-cN})$ for any $c<-\log q$.

	These are exactly the oscillatory terms present: the parity expansion contributes $(-1)^k k^{1/2-h}$ with $h\ge1$ (there is no $h=0$ term, since the leading $b_0=1$ is common to both parities), so the estimate above applies with $\alpha=h-\tfrac12\ge\tfrac12>0$. The exceptional uniform case $n=2$ ($\theta=1$) is excluded and treated separately.

	It remains to average the smooth powers of $K$. Let $\mu=N\theta$. Two events
	must first be stitched: the decomposition of $g(K)$ is used on
	$\{K\ge\mu/2\}$, while the oscillatory estimate above was proved on
	$\{K\ge1\}$; the difference is a term supported on $\{1\le K<\mu/2\}$, whose
	probability is exponentially small by Hoeffding's
	inequality~\cite[Theorem~2]{hoeffding1963}, so the two estimates combine. On the
	discarded event $\{K<\mu/2\}$ the integrand itself is bounded crudely by
	$g(k)\le k\le N$, so its total contribution is $O(Ne^{-cN})$, negligible at
	every order. On $K\ge\mu/2$, the remainder
	$O(K^{1/2-R-1})$ is $O(N^{1/2-R-1})$.

	For the smooth powers we claim: for each fixed exponent $\alpha$ occurring in
	the smooth part ($\alpha=\tfrac12-h$, $h=0,1,\dots,R$) and every fixed
	$L\ge0$,
	\begin{equation}\label{eq:taylor-claim}
		\E\bigl[K^\alpha;\,K\ge\mu/2\bigr]
		=
		\sum_{\ell=0}^{L}
		\binom{\alpha}{\ell}\,
		\mu^{\alpha-\ell}\,
		\E(K-\mu)^\ell
		+O\bigl(N^{\alpha-(L+1)/2}\bigr).
	\end{equation}
	This is the rigorous form of the Taylor step. Indeed, apply Taylor's theorem
	\emph{with remainder} to $f(x)=x^\alpha$ at $x=\mu$ on the interval
	$[\mu/2,N]$ (which contains $K$ on the event considered; note that the
	deviation $(K-\mu)/\mu$ need not stay within the radius of the binomial
	series, so the finite form is essential). Since $x\ge\mu/2$ there and
	$\alpha-L-1<0$,
	\[
		\sup_{\mu/2\le x\le N}\bigl|f^{(L+1)}(x)\bigr|
		\le C_{\alpha,L,\theta}\,\mu^{\alpha-L-1},
	\]
	so on $\{K\ge\mu/2\}$
	\[
		K^\alpha
		=\sum_{\ell=0}^{L}\binom{\alpha}{\ell}\,\mu^{\alpha-\ell}(K-\mu)^\ell
		+O\bigl(\mu^{\alpha-L-1}|K-\mu|^{L+1}\bigr).
	\]
	For a binomial variable with fixed parameter $\theta$ the standard
	central-moment bound (sums of independent bounded centered variables) gives
	\[
		\E|K-\mu|^{L+1}=O\bigl(N^{(L+1)/2}\bigr),
	\]
	so the expected Taylor remainder is
	$O(N^{\alpha-L-1}\,N^{(L+1)/2})=O(N^{\alpha-(L+1)/2})$. Finally, replacing the
	truncated moments $\E[(K-\mu)^\ell;K\ge\mu/2]$ by the full central moments
	$\E(K-\mu)^\ell$ changes the expression by at most $O(N^\ell e^{-cN})$,
	because \(K<\mu/2\) is exponentially unlikely and \(|K-\mu|^\ell=O(N^\ell)\)
	there. This proves \eqref{eq:taylor-claim}.

	Since the full central moment
	$\E(K-\mu)^\ell$ is a polynomial in $N$ of degree $\lfloor\ell/2\rfloor$,
	the average of each smooth term has an expansion in integer powers of $N^{-1}$
	relative to its leading order $N^\alpha$; the leading smooth term ($h=0$,
	$\alpha=\tfrac12$) gives
	$\sqrt{2/\pi}\,\E K^{1/2}=\sqrt{2N\theta/\pi}\,(1+O(N^{-1}))
	=2\sqrt{N/(n\pi)}\,(1+O(N^{-1}))$.
	Combining these finitely many contributions gives the asserted complete
	expansion of $W_N$. For the bookkeeping of the truncation parameters: given a
	target order $k_0$, choose $R,L\ge2k_0+2$; on $\{K\ge\mu/2\}$ the walk-side
	remainder is $O(N^{1/2-R-1})$ and the Taylor remainders are
	$O(N^{1/2-(L+1)/2})$, both beyond relative order $N^{-k_0}$; since the oscillating part and the small-$K$ region are exponentially small and $k_0$ is arbitrary, the expansion is complete; all implied constants depend on the fixed $n$.
\end{proof}

\begin{lemma}[Exact recurrence for $W_N$]
	\label{lem:exact-W-recurrence}
	For every $n\ge2$, the sequence $W_N=\E|X_2-X_1|$ (the per-edge expectation under the uniform $n$-bin law) satisfies
	\[
		n(N+1)W_{N+2}
		-2\bigl((n-2)(N+1)+1\bigr)W_{N+1}
		+(n-4)(N+1)W_N=0,
	\]
	with $W_1=2/n$ and $W_2=4(n-1)/n^2$. At $n=4$ the last coefficient vanishes
	and the recurrence degenerates to a two-term one (consistently with the closed
	form of Theorem~\ref{thm:Ed-n4}); at $n=2$ it reproduces the folded-binomial
	sequence of Theorem~\ref{thm:Ed-n2}. The asymptotic analysis below uses it for
	$n\ge3$.
\end{lemma}

\begin{proof}
	From the exact finite sum for the uniform case, with the abbreviation $a:=n-2$
	(the letter $z$ is reserved for the generating variable of
	\S\ref{ssec:trinomial}), set
	\[
		U_M:=\sum_{j=0}^{M}\binom{M}{j}
		\binom{j}{\lfloor j/2\rfloor}a^{M-j},
		\qquad
		W_N=\frac{2N}{n^N}U_{N-1},
	\]
	so that $U_{N-1}=S_N$, the sum of \S\ref{ssec:trinomial}. With the generating
	function
	$A(w)=\sum_{j\ge0}\binom{j}{\lfloor j/2\rfloor}w^j
	=\tfrac{1}{2w}\bigl(\sqrt{(1+2w)/(1-2w)}-1\bigr)$ of Step~1 in the proof of
	Theorem~\ref{thm:Ed-closed}, the ordinary generating function of $U_M$ is
	\[
		H(t):=\sum_{M\ge0}U_Mt^M
		=
		\frac{1}{1-at}A\!\left(\frac{t}{1-at}\right)
		=
		\frac{1}{2t}
		\left(\sqrt{\frac{1-(a-2)t}{1-(a+2)t}}-1\right).
	\]
	These are formal power-series identities. Write
	\begin{gather*}
		R(t):=\sqrt{\frac{1-(a-2)t}{1-(a+2)t}}=2tH(t)+1,\\
		P(t):=1-2at+(a^2-4)t^2=(1-(a-2)t)(1-(a+2)t).
	\end{gather*}
	Logarithmic differentiation gives $R'/R=2/P$, i.e.\ $P\,R'=2R$; substituting $R=2tH+1$,
	$R'=2H+2tH'$ and using $P-2t=1-2(a+1)t+(a^2-4)t^2$ gives
	\[
		t\bigl(1-2at+(a^2-4)t^2\bigr)H'(t)
		+\bigl(1-2(a+1)t+(a^2-4)t^2\bigr)H(t)-1=0.
	\]
	Substituting $H(t)=\sum_{M\ge0}U_Mt^M$ and comparing coefficients gives
	\[
		(M+1)(a^2-4)U_M
		-2\bigl((M+2)a+1\bigr)U_{M+1}
		+(M+3)U_{M+2}=0.
	\]
	Finally insert $M=N-1$ and
	\[
		U_{N-1}=\frac{n^N}{2N}W_N.
	\]
	After clearing the common factor, the displayed recurrence for $W_N$
	follows.
\end{proof}

\begin{remark}[The recurrence as an exact tabulator]\label{rem:exact-tabulation}
	Beyond its role in the asymptotic analysis, the recurrence of
	Lemma~\ref{lem:exact-W-recurrence} is a convenient route to the expectation
	\emph{exactly}. Seeded by its initial values $W_1,W_2$, it generates the whole
	exact sequence $W_1,\dots,W_K$ --- and hence $\E(d)=(n-1)W_N$ and
	$\E(d^\circ)=nW_N$ --- in $O(K)$ rational operations, a fixed number of integer
	multiplications and a single division per step, in contrast to evaluating the
	closed sum~\eqref{eq:Ed-sum} term by term ($\Theta(N)$ summands per value).
\end{remark}

\begin{lemma}[Recurrence for the asymptotic coefficients]
	\label{lem:e-recurrence}
	Set $e_0=1$. The coefficients in Lemma~\ref{lem:mixture-existence}
	satisfy, for $k\ge1$,
	\[
		e_k=\frac{1}{4k}\sum_{j=0}^{k-1}\gamma_{k,j}e_j,
	\]
	where, with $m=k+1-j$,
	\[
		\gamma_{k,j}
		=
		\bigl[n(2^m-2)+4\bigr]\binom{\tfrac12-j}{m}
		+
		\bigl[n(2^{m-1}-2)+2\bigr]\binom{\tfrac12-j}{m-1}.
	\]
\end{lemma}

\begin{proof}
	Because Lemma~\ref{lem:mixture-existence} guarantees that a complete expansion
	\emph{exists}, substituting the ansatz into the exact recurrence and matching
	coefficients is legitimate, by the uniqueness of asymptotic expansions: the
	recurrence computes the expansion that the lemma guarantees.
	Write the expansion from Lemma~\ref{lem:mixture-existence} in the form
	\[
		W_N=C\,N^{1/2}E(N),
		\qquad
		C=\frac{2}{\sqrt{n\pi}},
		\qquad
		E(N)\sim\sum_{j\ge0}e_jN^{-j}.
	\]
	Substitute this into the recurrence of Lemma~\ref{lem:exact-W-recurrence}
	and divide by the common factor $C N^{3/2}$. The coefficient multiplying
	a fixed $e_j$ is obtained by expanding
	\[
		n(1+N^{-1})(1+2N^{-1})^{1/2-j}
		-2\bigl[(n-2)+(n-1)N^{-1}\bigr](1+N^{-1})^{1/2-j}
		+(n-4)(1+N^{-1}).
	\]
	Let $x=N^{-1}$. Denote this multiplier by
	\[
		G_j(x)=
		n(1+x)(1+2x)^{1/2-j}
		-2\bigl[(n-2)+(n-1)x\bigr](1+x)^{1/2-j}
		+(n-4)(1+x).
	\]
	The leading terms cancel identically. After the division by $CN^{3/2}$,
	matching the coefficient of $N^{-(k+1)}$
	in the recurrence is the same as collecting, for each $j\le k$, the
	coefficient of $x^{k+1-j}$ in $G_j(x)$.

	For $j<k$, put $m=k+1-j\ge2$. The last term $(n-4)(1+x)$ of $G_j$ then contributes nothing. By $(1+cx)^{1/2-j}=\sum_i\binom{1/2-j}{i}(cx)^i$, the degree-$m$ pieces carry $\binom{1/2-j}{m}$ and the degree-$(m-1)$ pieces (times the extra factor $x$) carry $\binom{1/2-j}{m-1}$; collecting the two,
	\[
		[x^m]G_j(x)
		=
		\bigl[n(2^m-2)+4\bigr]\binom{\tfrac12-j}{m}
		+
		\bigl[n(2^{m-1}-2)+2\bigr]\binom{\tfrac12-j}{m-1}
		=\gamma_{k,j}.
	\]
	For $j=k$ ($m=1$) the linear term $(n-4)$ now contributes, and the same expansion gives $n(2-2k)-(n-2)(1-2k)-2(n-1)+(n-4)$, i.e.
	\[
		[x]G_k(x)=-4k.
	\]
	Thus the coefficient comparison has the form
	\[
		-4k\,e_k+\sum_{j=0}^{k-1}\gamma_{k,j}e_j=0,
	\]
	which proves the recurrence. The recursion is nonsingular at every step: the
	coefficient of $e_k$ is $-4k\ne0$, so each $e_k$ is determined uniquely by its
	predecessors.
\end{proof}

\subsection{Full asymptotic expansion of expectation}
	\label{ssec:full}

\begin{theorem}[Higher asymptotic expansion of the cyclic DTV]
	\label{thm:main-expansion}
	For every fixed $n\ge3$,
	\[
		\E\!\left(\frac{d^\circ}{\sqrt{Nn}}\right)
		\sim
		\frac{2}{\sqrt\pi}
		\left(1+\sum_{k\ge1}\frac{e_k(n)}{N^k}\right),
		\qquad N\to\infty,
	\]
	where $e_0=1$ and $e_k$ is generated by the recurrence of
	Lemma~\ref{lem:e-recurrence}. 
\end{theorem}

\begin{proof}
	By Lemma~\ref{lem:mixture-existence},
	\[
		W_N\sim
		2\sqrt{\frac{N}{n\pi}}
		\left(1+\sum_{k\ge1}\frac{e_k(n)}{N^k}\right).
	\]
	Since $\E(d^\circ)=nW_N$, division by $\sqrt{Nn}$ gives
	\[
		\E\!\left(\frac{d^\circ}{\sqrt{Nn}}\right)
		\sim
		\frac{2n}{\sqrt{Nn}}
		\sqrt{\frac{N}{n\pi}}
		\left(1+\sum_{k\ge1}\frac{e_k(n)}{N^k}\right)
		=
		\frac{2}{\sqrt\pi}
		\left(1+\sum_{k\ge1}\frac{e_k(n)}{N^k}\right).
	\]
	Lemma~\ref{lem:e-recurrence} gives the recurrence for the coefficients.
	Substituting $k=1,2,3,4$ into that recurrence gives the values
	$e_1,\ldots,e_4$ listed in Corollary~\ref{cor:coefficients}.
\end{proof}

\begin{corollary}[Coefficients of the asymptotic expansion]
	\label{cor:coefficients}
	With the notation of Theorem~\ref{thm:main-expansion}, the leading coefficients
	$e_k(n)$ are
	\[
	\begin{aligned}
		e_1&=-\frac{n-2}{16},\\
		e_2&=\frac{-3n^2+12n+4}{512},\\
		e_3&=-\frac{5(n-2)(3n^2-12n-4)}{8192},\\
		e_4&=-\frac{21(5n^2-20n+4)^2}{524288}.
	\end{aligned}
	\]
\end{corollary}

\begin{proof}
	Direct substitution of $k=1,2,3,4$ into the recurrence of Lemma~\ref{lem:e-recurrence}.
\end{proof}

\begin{remark}[Fixed $n$; the effective parameter is $n/N$]\label{rem:n-uniformity}
	All $N\to\infty$ expansions in this paper hold for \emph{fixed} $n$, with
	$n$-dependent implied constants; the expansion of
	Theorem~\ref{thm:main-expansion} is \emph{not} claimed uniform in $n$. Indeed,
	the recurrence shows that $e_k(n)$ is a polynomial in $n$ of degree $k$
	($e_1\sim-n/16$, $e_2\sim-3n^2/512$; note also the common factor $3n^2-12n-4$
	of $e_2$ and $e_3$), so the series is morally an expansion in $n/N$ and is
	numerically useful for $N\gg n$. The complementary sparse regime (fixed $N$,
	$n\to\infty$) is treated in a separate paper; the two regimes are expected to
	join through the Poisson limit, with $N/n$ fixed.
\end{remark}

	Table~\ref{tab:higher-terms} shows the gain from the higher coefficients at
	$n=10$, with $R_N:=\E(d^\circ/\sqrt{Nn})$; its ``two terms'' and ``four terms''
	columns are the truncations of Theorem~\ref{thm:main-expansion} keeping the first
	two and the first four corrections~$e_k$, respectively, and the exact column is
	again generated by the recurrence of Lemma~\ref{lem:exact-W-recurrence}
	(Remark~\ref{rem:exact-tabulation}).
	\begin{table}[ht]
	\centering\small
	\begin{tabular}{@{}r|c|c|c|c|c@{}}
		$N$ & exact $R_N$ & two terms & error & four terms & error \\\hline
		50   & $1.1169317$ & $1.1169402$ & $8.5\cdot10^{-6}$ & $1.1169318$ & $9.9\cdot10^{-8}$ \\
		100  & $1.1226975$ & $1.1226985$ & $1.0\cdot10^{-6}$ & $1.1226975$ & $2.8\cdot10^{-9}$ \\
		200  & $1.1255484$ & $1.1255485$ & $1.2\cdot10^{-7}$ & $1.1255484$ & $8.4\cdot10^{-11}$ \\
		2000 & $1.1280970$ & $1.1280970$ & $1.2\cdot10^{-10}$ & $1.1280970$ & $8.1\cdot10^{-16}$ \\
	\end{tabular}
	\caption{Two- and four-term truncations of Theorem~\ref{thm:main-expansion}
	against exact values, uniform law, $n=10$.}
	\label{tab:higher-terms}
	\end{table}

\subsection{The exceptional two-bin case}
	\label{ssec:two-bin_case}

	The case $n=2$ is genuinely exceptional because $K=X_1+X_2=N$
	deterministically, so the parity correction in Lemma~\ref{lem:mixture-existence}
	survives.
	We use the shifted central-binomial expansion of Proposition~\ref{prop:binomial-shift}.

\begin{theorem}[Complete two-bin expansion]\label{thm:Ed-n2}
	For $n=2$, $d=|X_2-X_1|=|2X_1-N|$ with $X_1\sim\Bin(N,\tfrac12)$, and the exact value
	(Lemma~\ref{lem:E-bin-walk}) is a single central binomial coefficient:
	\[
		\E(d)|_{n=2}=N\binom{N-1}{\lfloor(N-1)/2\rfloor}2^{-(N-1)}
		=N\binom{2\lfloor N/2\rfloor}{\lfloor N/2\rfloor}4^{-\lfloor N/2\rfloor}.
	\]
	With the shift $\nu:=\lfloor N/2\rfloor+\tfrac14=\dfrac{2N+(-1)^N}{4}$, the
	central-binomial expansion of Proposition~\ref{prop:binomial-shift}
	gives the complete asymptotic expansion
	\[
		\E(d)|_{n=2}\sim\frac{N}{\sqrt{\pi\nu}}\Bigl[1-\frac1{64\nu^2}+\frac{21}{8192\nu^4}
		-\frac{671}{524288\nu^6}+\frac{180323}{134217728\nu^8}-\cdots\Bigr],
	\]
	in the usual sense: truncating the bracket after the $\nu^{-2R}$ term leaves a relative error $O(\nu^{-2R-2})$ there, hence an absolute error $O(N\,\nu^{-2R-5/2})$ for $\E(d)$. The shift $\nu$ is optimal in that it annihilates \emph{all}
	odd powers of $\nu^{-1}$: the expansion proceeds in even powers only, and the
	parity of $N$ enters solely through the definition of $\nu$.
\end{theorem}

\begin{proof}
	The first equality is Lemma~\ref{lem:E-bin-walk}. To obtain the second, put
	$h=\lfloor N/2\rfloor$ and distinguish the two parities. If $N=2h$, then
	\[
		\binom{N-1}{\lfloor(N-1)/2\rfloor}
		=\binom{2h-1}{h-1}=\frac12\binom{2h}{h},
	\]
	and the factor $2^{-(N-1)}=2^{-(2h-1)}$ cancels the factor $1/2$, leaving
	$\binom{2h}{h}4^{-h}$. If $N=2h+1$, the same expression is obtained directly:
	\[
		\binom{N-1}{\lfloor(N-1)/2\rfloor}2^{-(N-1)}
		=\binom{2h}{h}2^{-2h}=\binom{2h}{h}4^{-h}.
	\]
	Thus the exact formula has the common form
	$\E(d)=N\binom{2h}{h}4^{-h}$ for both parities.

	Now $\nu=h+\tfrac14=(2N+(-1)^N)/4$. The shifted central-binomial expansion
	of Proposition~\ref{prop:binomial-shift}, specialized to the shift
	$\alpha=\tfrac14$, is
	\[
		\binom{2h}{h}4^{-h}
		\sim\frac1{\sqrt{\pi\nu}}\left[1-\frac1{64\nu^2}
		+\frac{21}{8192\nu^4}-\frac{671}{524288\nu^6}
		+\frac{180323}{134217728\nu^8}-\cdots\right].
	\]
	Multiplication by $N$ gives the stated complete expansion. The shifted
	form contains only even inverse powers of $\nu$ (the odd-index coefficients in
	Proposition~\ref{prop:binomial-shift} vanish), which is why the parity
	dependence is confined to the definition of $\nu$.
\end{proof}

\section{Concluding remarks}
	We have given a complete account of the first moment of the discrete total
	variation under the uniform multinomial null: an exact single sum, a
	closed form in central trinomial coefficients (valid for all $n\ge2$, with
	one-term evaluations at $n=2,4$ and a complete D-finite characterization for
	$n=3$), an exact three-term recurrence for the per-edge expectation $W_N$, and the
	full $1/N$ expansion with its $n$-independent constant $2/\sqrt\pi$. The complementary sparse
	regime (fixed $N$, $n\to\infty$) is treated in a separate paper.

	This is the first part of a series on the exact and asymptotic moments of
	discrete total variation under multinomial sampling. The mixture-conditioning
	method used here is the common thread: whereas the closed form of
	\S\ref{sec:Ed-exact} is tied to the uniform symmetry, conditioning on a two-bin
	total carries over to an arbitrary reference law. Part~2 treats the variance of
	$d$ under the uniform null; Part~3 the expectation under a general reference law
	$\mathbf r$, where this method generalizes; Part~4 the variance
	under a general reference law; and Part~5 the cyclic statistic $d^\circ$ as its subject proper --- its
	variance, the exact linear--cyclic transfer identity, and the large-bin variance
	ridge. Together, these moments form the \emph{moment backbone} of the null
	calibration of $d$; the asymptotic null laws themselves --- the fixed-$n$
	folded-Gaussian limit law, its expansions, and the large-$n$ normal
	approximation, which complement the exact dynamic-programming computation
	of~\cite{banic_elezovic_ct} --- are the subject of a companion paper, for which
	the moments computed in this series are the inputs.

\section*{Acknowledgments}
	The authors thank L.~Mihokovi\'c for pointing out the connection with the Legendre
	polynomials (Remark~\ref{rem:trinomial}).
	The authors used an AI assistant (Anthropic's Claude) for literature search,
	symbolic and numerical verification, and editorial support; all mathematical
	results were checked independently by the authors.

\end{document}